\documentclass[12pt]{article}

\usepackage{amssymb,a4}
\usepackage{amsmath,amsfonts,amssymb,amsthm, mathrsfs}

\newcommand{\ot}{\otimes}

\def\C{{\mathbb C}}

\def\Z{{\mathbb Z}}

\def\1{{\bf 1}}

\def \<{\langle}
\def \>{\rangle}

\def\ot{\otimes}

\def \bee{\begin{equation}\label}

\def\k{\kappa}

\def\bZ{{\mathbb Z}}

\def\bC{{\mathbb C}}

\def\C{{\mathscr C}}

\def\b{\mathfrak{b}}

\def\C{{\mathbb C}}

\def\Z{{\mathbb Z}}

\def\1{{\bf 1}}
\def\b1{\mathbbold{1}}
\def\k{{\bf k}}

\newtheorem{theorem}{Theorem}[section]
\newtheorem{prop}[theorem]{Proposition}
\newtheorem{lem}[theorem]{Lemma}
\newtheorem{coro}[theorem]{Corollary}
\newtheorem{rem}[theorem]{Remark}
\theoremstyle{definition}
\newtheorem{definition}[theorem]{Definition}

\begin{document}
\begin{center}
{\Large {\bf Whittaker  modules for the twisted affine Nappi-Witten
Lie algebra  $\widehat{H}_{4}[\tau]$  }}  \\
\vspace{0.5cm} Xue Chen\footnote{Supported in part by China NSF
grant 11771281 and  Fujian Province NSF Grant 2017J05016.} \\
 School of Applied Mathematics, Xiamen University of Technology\\
  Xiamen 361024, China\\
E-mail: xuechen@xmut.edu.cn\\
\vspace{0.5cm} Cuipo Jiang\footnote{Supported by China NSF grants
11531004, 11771281.}
\\
School of Mathematical Science,  Shanghai Jiao Tong University\\
 Shanghai, 200240, China \\
E-mail: cpjiang@sjtu.edu.cn
\end{center}

\begin{abstract}
  The Whittaker  module $M_{\psi}$  and   its quotient  Whittaker  module $L_{\psi, \xi}$ for the twisted affine Nappi-Witten Lie
algebra $\widehat{H}_{4}[\tau]$ are studied.   For nonsingular type,
it is proved that if  $\xi\neq 0$,  then $L_{\psi,\xi}$ is irreducible   and any irreducible Whittaker $\widehat{H}_{4}[\tau]$-module of type $\psi$
with ${\bf k}$ acting as a non-zero scalar $\xi$ is isomorphic to $L_{\psi,\xi}$. Furthermore,  for $\xi=0$, all Whittaker vectors
of $L_{\psi, 0}$ are completely determined. For singular type, the Whittaker vectors of $L_{\psi, \xi}$ with $\xi \neq 0$ are
fully characterized.
2000MSC:17B10, 17B65
\end{abstract}

\section{Introduction}

It is well known that the conformal field theory(CFT) has many
applications in various aspects of  mathematics and
 physics.
Wess-Zumino-Novikov-Witten (WZNW) models \cite{W} form one of the typical examples of CFTS. WZNW models were considered originally
within the framework of semisimple groups.  Later it was found that
the WZNW models based on non-abelian non-semisimple Lie groups are
closely relevant to the construction of exact string backgrounds.
For this reason, WZNW models of the special types have drawn much
research interest \cite{D,KK,NW,W}. In 1993, it was observed  by
Nappi and Witten  \cite{NW} that the WZNW model based on a central
extension of the two-dimensional Euclidean group describes the
homogeneous four-dimensional space-time corresponding to a
gravitational plane wave. The related Lie algebra, which is called
Nappi-Witten Lie algebra,  is neither abelian nor semisimple.

The Nappi-Witten Lie algebra $H_{4}$ is  a four-dimensional vector
space with $\C$-basis $\{a, b, c, d\}$ subject to the following
relations
$$
[a,b]=c,\ \ [d,a]=a,\ \ [d,b]=-b,\ \ [c,d]=[c,a]=[c,b]=0.
$$
Let $(\cdot, \cdot)$ be a symmetric bilinear form on $H_{4}$ defined
by
$$
(a,b)=1,\quad (c,d)=1,\quad\text{otherwise}, \ (\cdot, \cdot)=0.
$$
One can easily check  that $(\cdot, \cdot)$ is a non-degenerate
symmetric invariant  bilinear form on $H_{4}$. Just as the
non-twisted affine Kac-Moody Lie algebras given in \cite{K}, the
non-twisted affine Nappi-Witten Lie algebra $\widehat{H}_{4}$ is
defined by
\begin{equation*}
H_{4}\ot\bC[t,t^{-1}]\oplus\bC{\bf k},
\end{equation*}
with the   bracket relations
\begin{eqnarray*}
[x\ot t^m,y\ot t^n]=[x,y]\ot t^{m+n}+m(x,y)\delta_{m+n,0}{\bf k},
\quad [\widehat{H}_{4}, {\bf k}]=0,
\end{eqnarray*}
for $x,y\in H_{4}$ and $m,n\in\bZ$.

Define a linear map $\tau$ of $H_{4}$ by
 $$\tau a=b,  \;\;  \tau b=a,  \;\; \tau c=-c, \;\; \tau d=-d.$$
Clearly, $\tau^{2}=  {\rm id}_{H_{4}}$ and $\tau \in {\rm
Aut}(H_{4})$. It follows that $H_{4}$ becomes a ${\Z}_{2}$-graded
Lie algebra:
$$H_{4}=(H_{4})_{0}\oplus (H_{4})_{1},$$
where $$(H_{4})_{0}=\C(a+b), \ \ \ (H_{4})_{1}=\C(a-b)\oplus \C
c\oplus \C d.$$ Define a linear transformation $\hat{\tau}$ of
$\widehat{H}_{4}$ by
$$\hat{\tau}(h\otimes t^{m}+ s{\bf k})=(-1)^{m}\tau(h)\otimes t^{m}+ s{\bf k}$$
for $m\in\Z, h\in H_{4}$, and $s\in\C$. It is easy to check that
$\hat{\tau}$ is a Lie automorphism of $\widehat{H}_{4}$.  The
twisted affine  Nappi-Witten Lie algebra $\widehat{H}_{4}[\tau]$ is
\begin{eqnarray*}
&& \widehat{H}_{4}[\tau]= \{ u\in \widehat{H}_{4}\ | \ \hat{\tau}(u)=u\}\\
&&=\sum_{m\in \Z}\C(a+b)\otimes t^{2m}\oplus(\sum_{m\in \Z}(\C(a-b)+
\C c+ \C d)\otimes t^{2m+1})\oplus \C{\bf k},
\end{eqnarray*}
which is a subalgebra of $\widehat{H}_{4}$ fixed by $\hat{\tau}$.
In the following, for convenience, we shall denote $x\otimes t^{n}$
by $x(n)$.

 The representation theory for the non-twisted affine
Nappi-Witten Lie algebra $\widehat{H}_{4}$ has been well studied in
\cite{BJP}. The irreducible restricted modules for $\widehat{H}_{4}$
with some natural conditions have been classified and
 the extension of the vertex operator algebra
$V_{\widehat{H}_{4}}(l,0)$ by the even lattice $L$ has been
considered in \cite{JW}. Verma modules and vertex operator
representations for the twisted affine Nappi-Witten Lie algebra
$\widehat{H}_{4}[\tau]$ have also been  investigated in \cite{CJJ}.

 Whittaker modules were originally introduced  by D. Arnal and G. Pinczon
 \cite{AP}  in the  construction of a very vast family of
representations for $sl(2)$. A class of Whittaker modules for an
arbitrary finite-dimensional complex semisimple Lie algebra
$\mathfrak{g}$ were defined by B. Kostant in \cite{K1}. Whittaker
modules play a critical role in the   classification of all
irreducible $sl(2)$-modules. It was  illustrated in \cite{B} that
the irreducible modules for $sl(2)$ consist of  highest (lowest)
weight modules, Whittaker modules, and a third family obtained by
localization. Since  the construction of Whittaker modules depends on  the triangular decomposition of finite-dimensional
complex semisimple Lie algebras, it is reasonable to consider
Whittaker modules for other algebras with a triangular
decomposition. In the context of quantum groups,   Whittaker modules
for $U_{h}(\mathfrak{g})$ and $U_{q}(sl_{2})$ were investigated in
\cite{S} and \cite{O}. Recently, Whittaker modules for the affine Lie algebra $A_{1}^{(1)}$,
the Virasoro algebra, generalized Weyl algebras, non-twisted affine  Lie
algebras, and the Schr\"{o}dinger-Witt algebra as well as
 the twisted Heisenberg-Virasoro algebra were also considered in
\cite{ALZ, OW, BO, C, ZTL, LWZ}.

Inspired by the works mentioned above, P. Batra and V. Mazorchuk later generalized the ideas of
both Whittaker modules and the underlying categories to a broad class of Lie
algebras in \cite{BM}. Their framework allowed for a unified explanation of some important results
(such as Lemma \ref{lemma2.1} in this paper). Meanwhile, they formulated some conjectures on the form of
Whittaker vectors and the classification of irreducible Whittaker modules for
Lie algebras with triangular decompositions.
The aim of the present paper is to study  Whittaker
modules for  the twisted affine Nappi-Witten Lie algebra
$\widehat{H}_{4}[\tau]$. Some ideas we use come from \cite{ALZ,BM,OW, ZTL}.

Let $M$ be an $\widehat{H}_{4}[\tau]$-module and let $\psi:
\widehat{H}_{4}[\tau]^{(+)}\rightarrow \mathbb{C}$ be a Lie algebra
homomorphism. Recall that a non-zero vector $v\in M$ is called a Whittaker vector of type $\psi$ if
$xv=\psi(x)v$ for every $x\in \widehat{H}_{4}[\tau]^{(+)}$. An
$\widehat{H}_{4}[\tau]$-module $M$ is said to be a Whittaker module of type $\psi$
 if there is a type $\psi$ Whittaker vector $\omega\in M$ which
generates $M$. In this case  $\omega$ is called a cyclic Whittaker
vector. For $\xi\in\C$, denote by $L_{\psi,\xi}$ the quotient Whittaker module by the submodule generated by $({\bf k}-\xi)\omega$.

 For the non-singular type, that is, $\psi(c(1))\neq 0$, we  prove in Theorem \ref{thm3.7}  that   $L_{\psi,\xi}$ is irreducible if and only if $\xi\neq 0$,  and any irreducible Whittaker $\widehat{H}_{4}[\tau]$-module of type $\psi$ for which ${\bf k}$ acts by a non-zero scalar $\xi$ is isomorphic to $L_{\psi,\xi}$. This result together with Corollary \ref{coro3.8} confirms
the  Conjecture 33 and Conjecture 34 proposed in \cite{BM},   in the setting of the  twisted affine Nappi-Witten Lie algebra
$\widehat{H}_{4}[\tau]$.
  Furthermore, for the more challenging and interesting  case  that $\psi(c(1))\neq 0$ and $\xi=0$, we determine in Theorem \ref{thm 4.3}  all Whittaker vectors of $L_{\psi,0}$. It turns out that the proof of Theorem \ref{thm 4.3} is quite non-trivial.

 For the singular type  $\psi(c(1))=0$, which is also quite interesting,   we give in Theorem \ref{prop5.3} a full characterization of the  Whittaker vectors of $L_{\psi,\xi}$ for $\xi\neq 0$. Finally, for the identically zero case $\psi=0$, and $\xi\neq 0$,  by Theorem \ref{thm5.5} we give a filtration of the Whittaker module $L_{0,\xi}$  with the simple sctions given by  the Verma module $M(\psi,\xi)$  with  multiplicity  infinity.

The paper is organized as follows.  In Section 2,
Whittaker modules $M_{\psi}$ and $L_{\psi, \xi}$ for
$\widehat{H}_{4}[\tau]$ are constructed. In Section 3, the Whittaker
vectors of $M_{\psi}$ and $L_{\psi, \xi}$ for $\psi(c(1))\neq 0$  and $\xi\neq 0$ are
studied. In Section 4, the  Whittaker vectors of $L_{\psi,0}$ for $\psi(c(1))\neq 0$  and $\xi= 0$ are completely discussed. In the last section,  the  Whittaker vectors of $M_{\psi}$
and $L_{\psi, \xi}$  for $\psi(c(1))= 0$  and $\xi\neq 0$ , and relations to Verma modules are studied.

Throughout the paper, we  use $\C$, $\C^{*}$,  $\mathbb{N}$, $\Z$
and $\Z_{+}$ to denote the sets of the complex numbers, the non-zero
complex numbers, the non-negative integers, the integers and the
positive integers respectively.

\section{Preliminaries}
\setcounter{equation}{0}

We  first recall some general results on Whittaker modules of complex Lie algebras with a quasi-nilpotent
Lie subalgebra in \cite{ALZ,BM}. A Lie algebra $\mathfrak{n}$ is said to be $quasi$-$nilpotent$ provided that
 $\bigcap\limits_{k=0}^{\infty} \mathfrak{n}^{k}=0$,
where $\mathfrak{n}^{0}:=\mathfrak{n}$, $\mathfrak{n}^{k+1}:=[\mathfrak{n}^{k}, \mathfrak{n}]$ for any $k\in \mathbb{N}$.
Let $\mathfrak{g}$ be a complex Lie algebra and
$\mathfrak{n}$ be a quasi-nilpotent Lie subalgebra of $\mathfrak{g}$. Let $\psi: \mathfrak{n}\rightarrow \mathbb{C}$
be a Lie algebra homomorphism and $V$  be a $\mathfrak{g}$-module. A non-zero vector $v\in V$ is called a $Whittaker$ $vector$ $of$ $type$ $\psi$ if $xv=\psi(x)v$ for all $x\in \mathfrak{n}$. The module $V$ is said to be a $type$ $\psi$ $Whittaker$ $module$ for $\mathfrak{g}$ if it is generated by a type $\psi$ Whittaker vector. We say that $\mathfrak{n}$ acts on $V$ $locally$ $nilpotently$ \cite{ALZ}
if for any $v\in V$ there is $s\in \mathbb{N}$ depending on $v$ such that $x_{1}x_{2}\cdots x_{s}v=0$ for any
$x_{1},x_{2},\cdots, x_{s}\in \mathfrak{n}$. Let $\mathfrak{n}^{(\psi)}=\{x-\psi(x) \; \mid \; x\in \mathfrak{n} \}$.

\vskip 0.3cm
The following result comes from Lemma 3.1 in \cite{ALZ} and Proposition 32 in \cite{BM}.
\begin{lem}\label{lemma2.1}{\rm (\cite{ALZ,BM})}
Let $V$ be a type $\psi$ Whittaker module for $\mathfrak{g}$. Suppose that
the adjoint action of $\mathfrak{n}$ on $\mathfrak{g}/\mathfrak{n}$ is locally nilpotent.
Then
\begin{itemize}
\item[{\rm(i)}] $\mathfrak{n}^{(\psi)}$ acts locally nilpotently on $V$. In particular, $x-\psi(x)$ acts locally nilpotently on $V$ for any $x\in \mathfrak{n}$;
\item[{\rm(ii)}] any nonzero submodule of $V$ contains a Whittaker vector of type $\psi$;
\item[{\rm(iii)}] if the vector space of Whittaker vectors of $V$ is one-dimensional, then $V$ is simple.
\end{itemize}
\end{lem}
Furthermore, we have
\begin{lem}\label{lemma2.2}
Let $V$ be a type $\psi$ Whittaker module for $\mathfrak{g}$. Suppose that
the adjoint action of $\mathfrak{n}$ on $\mathfrak{g}/\mathfrak{n}$ is locally nilpotent.
Then Whittaker vectors in V are all of type $\psi$.
\end{lem}
\begin{proof}
Let $\psi': \mathfrak{n}\rightarrow \mathbb{C}$ be a
Lie algebra homomorphism and $\psi'\neq \psi$. Assume $v\in V$ is
 a  Whittaker vector of type $\psi'$. Then $(y-\psi'(y))v=0$, for any $y\in \mathfrak{n}$.
From (i) we know that there is $s\in \mathbb{Z}_{+}$ such that
\begin{eqnarray*}
(x_{1}-\psi(x_{1}))(x_{2}-\psi(x_{2}))\cdots (x_{s}-\psi(x_{s}))v=0, \ \text {for any}\; x_{1}, x_{2},\cdots, x_{s}\in \mathfrak{n}.
\end{eqnarray*}
Take $s$ minimal. Then there exist $x_{2},x_{3},\cdots, x_{s}\in \mathfrak{n}$ such that
\begin{eqnarray*}
v'=(x_{2}-\psi(x_{2}))\cdots (x_{s}-\psi(x_{s}))v\neq0
\end{eqnarray*}
 and
 \begin{eqnarray*}
 (x-\psi(x))v'=0, \ \text {for all}\; x\in \mathfrak{n}.
 \end{eqnarray*}
 So $v'$ is a Whittaker vector of type $\psi$.
On the other hand, note that $\psi'([\mathfrak{n},\mathfrak{n}])=0$. By inductive assumption, we deduce that
\begin{eqnarray*}
&&(y-\psi'(y))v'\\
&=&(y-\psi'(y))(x_{2}-\psi(x_{2}))\cdots (x_{s}-\psi(x_{s}))v\\
&=&\sum_{i=2}^{s}(x_{2}-\psi(x_{2}))\cdots (x_{i-1}-\psi(x_{i-1}))ady(x_{i})(x_{i+1}-\psi(x_{i+1}))\cdots (x_{s}-\psi(x_{s}))v\\
&=&0
\end{eqnarray*}
 for any $y\in \mathfrak{n}$. That is $v'$ is also a Whittaker vector of type $\psi'$.
Then $\psi'= \psi$,  a contradiction to our assumption. Thus the lemma holds.
\end{proof}

 The twisted affine Nappi-Witten Lie algebra
$\widehat{H}_{4}[\tau]$  has a natural decomposition:
\begin{eqnarray*}
\widehat{H}_{4}[\tau]=\widehat{H}_{4}[\tau]^{(+)}\bigoplus
\widehat{H}_{4}[\tau]^{(0)} \bigoplus \widehat{H}_{4}[\tau]^{(-)},
\end{eqnarray*}
where
\begin{eqnarray*}
&&\widehat{H}_{4}[\tau]^{(+)}=Span_{\mathbb{C}}\{(a+b)(n), (a-b)(m),
c(m), d(m)\; |\;   n\in 2\mathbb{Z}_{+}, m\in 2\mathbb{N}+1\},\\
&&\widehat{H}_{4}[\tau]^{(-)}=Span_{\mathbb{C}}\{(a+b)(-n),
(a-b)(-m),
c(-m), d(-m)\; |\;   n\in 2\mathbb{Z}_{+}, m\in 2\mathbb{N}+1\},\\
&&\widehat{H}_{4}[\tau]^{(0)}=Span_{\mathbb{C}}\{(a+b)(0), {\bf
k}\}.
\end{eqnarray*}
Denote
\begin{eqnarray*}
\mathfrak{b}^{-}=\widehat{H}_{4}[\tau]^{(-)}\oplus
\widehat{H}_{4}[\tau]^{(0)}.
\end{eqnarray*}
Let $\mathbb{C}[{\bf k}]$ be the polynomial algebra generated
by ${\bf k}$. It is easy to see that $\mathbb{C}[{\bf k}]$ lies in
 the center of the universal enveloping algebra $U(\widehat{H}_{4}[\tau])$.

The following notation for $odd$ $partitions$ and $even$
$pseudopartitions$ will be used to describe bases for
$U(\widehat{H}_{4}[\tau])$ and  Whittaker modules.

Just as in \cite{OW, ZTL}, for a non-decreasing sequence of
positive odd numbers:
 \begin{eqnarray*}
 0<\lambda_{1}\leq \lambda_{2}\leq\cdots
 \leq\lambda_{t},\;  \lambda_{i}\in 2\mathbb{N}+1, \; 1\leq i\leq
 t,
 \end{eqnarray*}
we call $\lambda=(\lambda_{1}, \lambda_{2},\cdots,  \lambda_{t})$ an
$odd$ $partition$. For a non-decreasing sequence of non-negative
even numbers:
\begin{eqnarray*}0\leq\mu_{1}\leq \mu_{2}\leq\cdots \leq
\mu_{s},\;  \mu_{i}\in 2\mathbb{N}, \; 1\leq i\leq
 s,
\end{eqnarray*}
 we call
$\tilde{\mu}=(\mu_{1}, \mu_{2},\cdots,  \mu_{s})$ an $even$
$pseudopartition$. Let $\mathcal {P}_{odd}$  and $\widetilde{\mathcal {P}}_{even}$ denote the set of
$odd$ $partitions$ and $even$ $pseudopartitions$, respectively.
For $\lambda\in \mathcal
{P}_{odd}$, $\tilde{\mu}\in \widetilde{\mathcal {P}}_{even}$, we
also write
\begin{eqnarray*}
\lambda=(1^{\lambda(1)}, 3^{\lambda(3)},\cdots),\;
\tilde{\mu}=(0^{\mu(0)}, 2^{\mu(2)},\cdots),
\end{eqnarray*}
where $\lambda(k)$ and $\mu(l)$ are the number of times of $k$ and
$l$ appear respectively in the  $odd$ $partition$ and $even$
$pseudopartition$, and $\lambda(k)=\mu(l)=0$ for $k$ and $l$
sufficiently large.

Let $\tilde{\mu}=(\mu_{1}, \mu_{2},\cdots,  \mu_{s})\in
\widetilde{\mathcal {P}}_{even}$. Let $\lambda=(\lambda_{1},
\lambda_{2},\cdots,  \lambda_{t})$, $\nu=(\nu_{1}, \nu_{2},\cdots,
\nu_{r})$, $\eta=(\eta_{1}, \eta_{2},\cdots, \eta_{l})$ and
$\lambda,\nu, \eta\in
\mathcal {P}_{odd}$.  We define
\begin{eqnarray*}
|\tilde{\mu}|=\mu_{1}+ \mu_{2}+\cdots + \mu_{s}, \quad |\lambda|
=\lambda_{1}+ \lambda_{2}+\cdots + \lambda_{t},
\end{eqnarray*}
\begin{eqnarray*}
\#(\tilde{\mu})=\mu(0)+ \mu(2)+ \cdots, \quad \#(\lambda)
=\lambda(1)+ \lambda(3)+\cdots,
\end{eqnarray*}
\begin{eqnarray*}
\#(\tilde{\mu},\nu,\lambda,\eta)=\#(\tilde{\mu})+\#({\nu})+\#({\lambda})
+\#({\eta}),
\end{eqnarray*}
\begin{eqnarray*}
(a+b)(-\tilde{\mu})=(a+b)(-{\mu}_{1})(a+b)(-{\mu}_{2})\cdots (a+b)(-{\mu}_{s})
=(a+b)(0)^{\mu(0)}(a+b)(-2)^{\mu(2)} \cdots,
\end{eqnarray*}
\begin{eqnarray*}
(a-b)(-\nu)=(a-b)(-{\nu}_{1})(a-b)(-{\nu}_{2})\cdots (a-b)(-{\nu}_{r})
=(a-b)(-1)^{\nu(1)}(a-b)(-3)^{\nu(3)}\cdots,
\end{eqnarray*}
\begin{eqnarray*}
d(-\lambda)=d(-{\lambda}_{1})d(-{\lambda}_{2})\cdots d(-\lambda_{t})
=d(-1)^{\lambda(1)}d(-3)^{\lambda(3)}\cdots,
\end{eqnarray*}
\begin{eqnarray*}
c(-\eta)=c(-{\eta}_{1})c(-{\eta}_{2})\cdots c(-\eta_{l})=c(-1)^{\eta(1)}c(-3)^{\eta(3)}\cdots.
\end{eqnarray*}
For convenience, we define $\bar{0}=(0^{0}, 1^{0}, 2^{0},\cdots)$
and set $(a+b)(\bar{0})=(a-b)(\bar{0})=d(\bar{0})=c(\bar{0})=1 \in
U(\widehat{H}_{4}[\tau])$. In what follows, we regard $\bar{0}$ as
an element of $\mathcal {P}_{odd}$ and  $\widetilde{\mathcal
{P}}_{even}$.

\begin{definition}(\cite{ALZ,OW,ZTL})
Let $M$ be an $\widehat{H}_{4}[\tau]$-module and let $\psi:
\widehat{H}_{4}[\tau]^{(+)}\rightarrow \mathbb{C}$ be a Lie algebra
homomorphism. A non-zero vector $v\in M$ is called a Whittaker vector of type $\psi$ if
$xv=\psi(x)v$ for every $x\in \widehat{H}_{4}[\tau]^{(+)}$. An
$\widehat{H}_{4}[\tau]$-module $M$ is said to be a Whittaker module of type $\psi$
 if there is a type $\psi$ Whittaker vector $\omega\in M$ which
generates $M$. In this case we call $\omega$ a cyclic Whittaker
vector.
\end{definition}

The Lie algebra homomorphism $\psi$ is called nonsingular if
$\psi(c(1))\neq 0$, otherwise $\psi$ is called singular. The
commutator relation in the definition of $\widehat{H}_{4}[\tau]$
forces that
$$\psi((a+b)(n))=\psi((a-b)(m))=\psi(c(m))=0,$$
for $n\in 2\mathbb{Z}_{+}$, $m\in 2\mathbb{Z}_{+}+1$.

Fix a Lie algebra homomorphism $\psi:
\widehat{H}_{4}[\tau]^{(+)}\rightarrow \mathbb{C}$, and let
$\mathbb{C}_{\psi}$ be  the one-dimensional
$\widehat{H}_{4}[\tau]^{(+)}$-module for which
$x\alpha=\psi(x)\alpha$ for $x\in\widehat{H}_{4}[\tau]^{(+)}$ and
$\alpha\in \mathbb{C}^{*}$. Then  an induced
$\widehat{H}_{4}[\tau]$-module is defined by
\begin{eqnarray}
M_{\psi}=U(\widehat{H}_{4}[\tau])\otimes
_{U(\widehat{H}_{4}[\tau]^{(+)})}\mathbb{C}_{\psi}.
\end{eqnarray}
For $\xi\in \mathbb{C}$, since ${\bf k}$ is in the center of
$\widehat{H}_{4}[\tau]$, $({\bf k}-\xi)M_{\psi}$ is a submodule of
$M_{\psi}$. Define
\begin{eqnarray}
L_{\psi,\xi}=M_{\psi}/({\bf k}-\xi)M_{\psi},
\end{eqnarray}
and let $\overline{ \; \cdot \; }: M_{\psi}\rightarrow L_{\psi,\xi}$ be the canonical homomorphism.
Then $L_{\psi,\xi}$ is a quotient module of $M_{\psi}$ for
$\widehat{H}_{4}[\tau]$. We have the following results for  $M_{\psi}$  and $L_{\psi,\xi}$.
\begin{prop}\label{prop2.1}
 \begin{itemize}
\item[{\rm (i)}] $M_{\psi}$ and $L_{\psi,\xi}$ are both   Whittaker modules of type $\psi$, with  cyclic Whittaker
vectors $\omega:=1\otimes1$ and $\bar{\omega}:=\overline{1\otimes1}$, respectively;
\item[{\rm (ii)}]  The set
\begin{eqnarray}
\{{\bf k}^{t}(a+b)(-\tilde{\mu})(a-b)(-\nu)d(-\lambda)
c(-\eta){\omega} |\;
 (\tilde{\mu},\nu,\lambda,\eta)                 \nonumber\\
\in \widetilde{\mathcal {P}}_{even}\times \mathcal {P}_{odd}\times
\mathcal {P}_{odd}\times \mathcal {P}_{odd},\; t\in \mathbb{N}\}
\end{eqnarray}
forms a basis of $M_{\psi}$.   $L_{\psi,\xi}$ has a basis
\begin{eqnarray}
\{(a+b)(-\tilde{\mu})(a-b)(-\nu)d(-\lambda)
c(-\eta)\bar{\omega} |\;
 (\tilde{\mu},\nu,\lambda,\eta)                 \nonumber\\
\in \widetilde{\mathcal {P}}_{even}\times \mathcal {P}_{odd}\times
\mathcal {P}_{odd}\times \mathcal {P}_{odd} \};
\end{eqnarray}
\item[{\rm (iii)}] $M_{\psi}$ has the universal property in the sense that
for any Whittaker module $M$ of type $\psi$ generated by $\omega'$,
there is a surjective module homomorphism $\varphi: M_{\psi}\rightarrow M$
such that $u\omega\mapsto u\omega'$, for  $ u\in
U(\mathfrak{b}^{-})$.  $M_{\psi}$ is called the universal
Whittaker module of type $\psi$;
\item[{\rm (iv)}] $L_{\psi,\xi}$ has the universal property in the sense that  for any Whittaker module $\bar{M}$ of $\widehat{H}_{4}[\tau]$ of type $\psi$ generated by a Whittaker vector $\bar{\omega'}$, such that ${\bf k}$ acts as the scalar $\xi$,
there is a surjective module homomorphism $\bar{\varphi}: L_{\psi, \xi}\rightarrow \bar{M}$
such that $u\bar{\omega}\mapsto u\bar{\omega'}$, for  $ u\in
U(\widehat{H}_{4}[\tau]^{(-)}\oplus \mathbb{C}(a+b)(0))$.
\end{itemize}
\end{prop}
\begin{proof} (i)-(iii) is obvious. We only prove (iv).  Let $\bar{M}$ be a Whittaker module of $\widehat{H}_{4}[\tau]$ of type $\psi$ generated by a cyclic  Whittaker vector $\bar{\omega'}$,  such that ${\bf k}$ acts as the scalar $\xi$. By (iii), there is a surjective module homomorphism $\varphi$ from $M_{\psi}$ to $\bar{M}$ such that $\varphi(\omega)=\bar{\omega'}$. Since on $\bar{M}$, ${\bf k}$ acts as the scalar $\xi$, we know that $\varphi( M_{\psi}({\bf k}-\xi)\omega)=0$. So $\varphi$ induce a surjective module homomorphism $\bar{\varphi}$ from $L_{\psi,\xi}$ to $\bar{M}$ such that $\bar{\varphi}(\bar{\omega})=\bar{\omega'}$.
\end{proof}

\begin{rem}
For any $x\in \widehat{H}_{4}[\tau]^{(+)}$,
$\omega'=u\bar{\omega}$,  $u\in U(\widehat{H}_{4}[\tau]^{(-)}\oplus \mathbb{C}(a+b)(0))$, we have
\begin{eqnarray*}
 &&(x-\psi(x))\omega'=[x, u]\bar{\omega}.
\end{eqnarray*}
\end{rem}

\begin{lem}\label{lem2.2}
For $n\in 2\mathbb{Z}_{+}, m\in 2\mathbb{N}+1$,
$\tilde{\mu}=(\mu_{1}, \mu_{2},\cdots,  \mu_{s})\in\tilde{\mathcal {P}}_{even}$,
  $\nu=(\nu_{1}, \nu_{2},\cdots, \nu_{r})\in
\mathcal {P}_{odd}$, the following formulas hold.
\begin{eqnarray}
(a+b)(n)(a+b)(-\tilde{\mu})=\sum_{i=1}^{s} 2n\delta_{n,\mu_{i}}{\bf
k}(a+b)(-\tilde{\mu}'^{(i)})+(a+b)(-\tilde{\mu})(a+b)(n),
\end{eqnarray}
where $|\tilde{\mu}'^{(i)}|=|\tilde{\mu}|-n$ and $\#(\tilde{\mu}'^{(i)})<\#(\tilde{\mu})$;
\begin{eqnarray}
(a+b)(n)(a-b)(-\nu)=-2\sum_{i=1}^{r}(a-b)(-\nu'^{(i)})c(n_{i})+(a-b)(-\nu)(a+b)(n),
\end{eqnarray}
where $|\nu'^{(i)}|-n_{i}=|\nu|-n$, $n_{i}< n$ and $n_{i}\in
2\mathbb{Z}+1$, $\#(\nu'^{(i)})<\#(\nu)$;

\begin{eqnarray}
(a-b)(m)(a+b)(-\tilde{\mu})=2\sum_{i=1}^{s}(a+b)(-\tilde{\mu}'^{(i)})c(m_{i})+(a+b)(-\tilde{\mu})(a-b)(m),
\end{eqnarray}
where $|\tilde{\mu}'^{(i)}|-m_{i}=|\tilde{\mu}|-m$, $m_{i}\leq m$
($m_{i}\in 2\mathbb{Z}+1$), and ${\mu}'^{(i)}(0)<\mu(0)$ if $m_{i}=
m$, $\#(\tilde{\mu}'^{(i)})<\#(\tilde{\mu})$;
\begin{eqnarray}
(a-b)(m)(a-b)(-\nu)=-\sum_{i=1}^{r} 2m\delta_{m,\nu_{i}}{\bf
k}(a-b)(-\nu'^{(i)})+(a-b)(-\nu)(a-b)(m),
\end{eqnarray}
where $|\nu'^{(i)}|=|\nu|-m$ and $\#(\nu'^{(i)})<\#(\nu)$.
\end{lem}

\begin{proof} We give a proof of (2.5) by induction on  $\#(\tilde{\mu})$. The proof for the rest formulas are similar.
Let $\#(\tilde{\mu})=1$ and $\tilde{\mu}=(\mu_{1})$. Then
\begin{eqnarray*}
(a+b)(n)(a+b)(-\tilde{\mu})= 2n\delta_{n,\mu_{1}}{\bf
k}+(a+b)(-\tilde{\mu})(a+b)(n).
\end{eqnarray*}
In this case $\#(\tilde{\mu}'^{(i)})=0$ and $(a+b)(-\tilde{\mu}'^{(i)})=1$ in (2.5).

Suppose that (2.5) holds for $\#(\tilde{\mu})< s$. Let $\#(\tilde{\mu})=s$ and $\tilde{\mu}=(\mu_{1}, \mu_{2},\cdots,  \mu_{s})$. We denote $(a+b)(-\tilde{\mu})=(a+b)(-\tilde{\mu}')(a+b)(-\mu_{s})$,
where $\tilde{\mu}'=(\mu_{1}, \mu_{2},\cdots,  \mu_{s-1})$.  Then
\begin{eqnarray*}
&&(a+b)(n)(a+b)(-\tilde{\mu})\\
&=&[\sum_{i=1}^{s-1} 2n\delta_{n,\mu_{i}}{\bf k}(a+b)(-\tilde{\mu}''^{(i)})+(a+b)(-\tilde{\mu}')(a+b)(n)](a+b)(-\mu_{s})\\
&=& \sum_{i=1}^{s-1} 2n\delta_{n,\mu_{i}}{\bf k}(a+b)(-\tilde{\mu}''^{(i)})(a+b)(-\mu_{s}) +  2n\delta_{n,\mu_{s}}{\bf k}(a+b)(-\tilde{\mu}')+(a+b)(-\tilde{\mu})(a+b)(n)\\
&=& \sum_{i=1}^{s} 2n\delta_{n,\mu_{i}}{\bf
k}(a+b)(-\tilde{\mu}'^{(i)})+(a+b)(-\tilde{\mu})(a+b)(n),
\end{eqnarray*}
where $(a+b)(-\tilde{\mu}''^{(i)})=(a+b)(-\mu_{1})\cdots \widehat{(a+b)(-\mu_{i})}\cdots (a+b)(-\mu_{s-1})$,
$(a+b)(-\tilde{\mu}'^{(i)})=(a+b)(-\mu_{1})\cdots \widehat{(a+b)(-\mu_{i})}\cdots(a+b)(-\mu_{s-1})(a+b)(-\mu_{s})$
and $\widehat{(a+b)(-\mu_{i})}$ means this factor is deleted.
 (2.5) is proved.
\end{proof}

\section{Whittaker vectors in  $L_{\psi,\xi}$ and $M_{\psi}$ for $\psi(c(1))\neq 0$ and $\xi\neq 0$}
\setcounter{equation}{0}

In this section,  it is  always assumed that $\psi(c(1))\neq 0$,
meaning that $\psi$ is nonsingular. Let $M_{\psi}$ and $L_{\psi,
\xi}$ defined by (2.1) and (2.2) be the Whittaker modules for the
twisted affine Nappi-Witten Lie algebra $\widehat{H}_{4}[\tau]$. The
key results of this section are shown in Theorem \ref{thm3.7} and Corollary \ref{coro3.8}
in which the Whittaker  vectors in $L_{\psi, \xi}$ and $M_{\psi}$
are characterized.

Applying Lemma \ref{lemma2.2} to $\widehat{H}_{4}[\tau]$, we have
\begin{lem}
Let $\psi(c(1))\neq 0$  and $L_{\psi,\xi}$ be a   Whittaker module for
$\widehat{H}_{4}[\tau]$ defined by (2.2). Then  the Whittaker
vectors in $L_{\psi,\xi}$ are all of type $\psi$.
\end{lem}

We are now in a position to give the main result of this section.
\begin{theorem}\label{thm3.7}
Assume that  $\psi(c(1))\neq 0$ and  $\xi \in \mathbb{C}^{*}$.  Let
$\bar{\omega}=\overline{1\otimes 1}\in L_{\psi, \xi}$.  Then

{\rm (1)} $\omega'\in
L_{\psi, \xi}$ is a Whittaker vector if and only if
$\omega'=u\bar{\omega}$ for some $u\in \mathbb{C}^{*}$;

{\rm (2)}   $L_{\psi,\xi}$ is irreducible;

{\rm (3)}   any  Whittaker $\widehat{H}_{4}[\tau]$-module of type $\psi$ for which $\psi(c(1))\neq 0$ and  ${\bf k}$ acts by a non-zero scalar $\xi\in \C$ is irreducible and isomorphic to $L_{\psi,
\xi}$.
\end{theorem}

\begin{proof} (1)
It is clear that if $\omega'=u\bar{\omega}$ for some $u\in \mathbb{C}^{*}$, then $\omega'$ is a Whittaker vector.
We now  prove the necessity.
 Note that for $m, n\in \mathbb{Z}_{+}$,
\begin{eqnarray*}
[c(2m+1),d(-2n-1)]=(2m+1)\delta_{m,n}{\bf k}.
\end{eqnarray*}
Then
 \begin{eqnarray*}
 {\frak s}=\mathop{\bigoplus}\limits_{m\in \mathbb{Z}_{+}}(\mathbb{C} c(2m+1)\oplus\mathbb{C}
d(-2m-1))\oplus \mathbb{C}{\bf k}
\end{eqnarray*}
is a Heisenberg algebra and $L_{\psi, \xi}$ can be viewed as an ${\frak s}$-module
on which ${\bf k}$ acts as $\xi \neq 0$.
Since every highest weight ${\frak s}$-module
generated by one element with $\k$ acting as a non-zero scalar is
irreducible, it follows  that $L_{\psi, \xi}$ can be decomposed into a direct
sum of irreducible highest weight modules of ${\frak s}$ with
the highest weight vectors
 \begin{eqnarray*}
\{(a+b)(-\tilde{\mu})(a-b)(-\nu)c(-\eta)d(-1)^{t}\bar{\omega}\; | \;  (\tilde{\mu},\nu,\eta)\in \widetilde{\mathcal
{P}}_{even}\times \mathcal {P}_{odd}\times \mathcal {P}_{odd}, t\in \mathbb{N} \}.
\end{eqnarray*}
Thus if $\bar{\omega}\neq \omega' \in L_{\psi, \xi}$ is a Whittaker
vector, then $\omega'$ can be written as
\begin{eqnarray*}
&\omega'=&\sum_{i\in I} x_{i} (a+b)(-\tilde{\mu}^{(i)})(a-b)(-\nu^{(i)})c(-\eta^{(i)})d(-1)^{t_{i}}\bar{\omega},
\end{eqnarray*}
where $I$ is a finite index set, $x_{i} \in \mathbb{C}^{*}$.

If there  exists $i\in I$ such that $t_{i} > 0$,  by considering the action of $c(1)$ on $\omega'$, we obtain
\begin{eqnarray*}
&&(  c(1)-\psi(c(1)) )\omega'\\
&=& \sum_{i\in I} x_{i} t_{i} \xi (a+b)(-\tilde{\mu}^{(i)})(a-b)(-\nu^{(i)})c(-\eta^{(i)})d(-1)^{t_{i}-1}\bar{\omega}\neq 0,
\end{eqnarray*}
a contradiction.  Then
\begin{eqnarray*}
&\omega'=&\sum_{i\in I} x_{i} (a+b)(-\tilde{\mu}^{(i)})(a-b)(-\nu^{(i)})c(-\eta^{(i)})\bar{\omega},
\end{eqnarray*}
where $I$ is a finite index set, $x_{i} \in \mathbb{C}^{*}$.
For $i\in I$, let
\begin{eqnarray*}
&& \tilde{\mu}^{(i)}=(0^{a_{i0}},2^{a_{i1}},\cdots, (2n)^{a_{in}} ),\\
&& \nu^{(i)}=(1^{b_{i0}},3^{b_{i1}},\cdots, (2m+1)^{b_{im}} ),\\
&& \eta^{(i)}= (1^{c_{i0}},3^{c_{i1}},\cdots, (2l+1)^{c_{il}} ),
\end{eqnarray*}
where  $n,m,l,a_{ij},b_{ij},c_{ij}\in \mathbb{N}$.
Then $\omega'$ can be written as
\begin{eqnarray*}
&\omega'=&\sum_{i\in I} x_{i} (a+b)(0)^{a_{i0}}(a+b)(-2)^{a_{i1}}\cdots (a+b)(-2n)^{a_{in}} \cdot\\
&&(a-b)(-1)^{b_{i0}}(a-b)(-3)^{b_{i1}} \cdots (a-b)(-2m-1)^{b_{im}} c(-1)^{c_{i0}}c(-3)^{c_{i1}}\cdots c(-2l-1)^{c_{il}}\bar{\omega},
\end{eqnarray*}
where $I$ is a finite index set, $n,m,l,a_{ij},b_{ij},c_{ij}\in \mathbb{N}$.

{\bf Case I}  Suppose $n\geq 0$ and $\{ a_{in} \; | \;  i\in I\}\neq \{0\}$.

If $m\geq n$ and there exists $i\in I$ such that $b_{im}\neq 0$,  by (2.5) and (2.6), we have
\begin{eqnarray*}
&&((a+b)(2m+2)-\psi((a+b)(2m+2)))\omega'\\
&=&\sum_{i\in I} x_{i}( -2\psi(c(1)) )b_{im}  (a+b)(0)^{a_{i0}}(a+b)(-2)^{a_{i1}}\cdots (a+b)(-2n)^{a_{in}}\cdot\\
 &&(a-b)(-1)^{b_{i0}}(a-b)(-3)^{b_{i1}}\cdots (a-b)(-2m-1)^{b_{im}-1} c(-\eta^{(i)})\bar{\omega}\neq  0,
\end{eqnarray*}
 a contradiction.

If $m< n$ and there exists $i\in I$ such that $b_{im}\neq 0$, by (2.7) and (2.8), we obtain
\begin{eqnarray*}
&&((a-b)(2n+1)-\psi((a-b)(2n+1)))\omega'\\
&=&\sum_{i\in I} x_{i}2\psi(c(1)) a_{in}  (a+b)(0)^{a_{i0}}(a+b)(-2)^{a_{i1}}\cdots (a+b)(-2n)^{a_{in}-1}\cdot\\
 &&(a-b)(-1)^{b_{i0}}(a-b)(-3)^{b_{i1}}\cdots (a-b)(-2m-1)^{b_{im}-1} c(-\eta^{(i)})\bar{\omega}\neq  0,
\end{eqnarray*}
 a contradiction.

{\bf Case II}  Suppose $m\geq 0$ and $\{ b_{im} \; | \;  i\in I\}\neq \{0\}$ and $\{ a_{in} \; | \;  i\in I\}= \{0\}$ for any $n\geq 0$.

Then $v$ has the following form
\begin{eqnarray*}
\omega'&=&\sum_{i\in I} x_{i}
(a-b)(-1)^{b_{i0}}(a-b)(-3)^{b_{i1}} \cdots (a-b)(-2m-1)^{b_{im}}
c(-\eta^{(i)})\bar{\omega}.
\end{eqnarray*}
By considering the action of $(a+b)(2m+2)$ on $\omega'$. We also get a contradiction.

{\bf Case III}  Suppose $l\geq 0$ and $\{ c_{il} \; | \;  i\in I\}\neq \{0\}$ and $\{ a_{in} \; | \;  i\in I\}=\{ b_{im} \; | \;  i\in I\}= \{0\}$ for any $n, m\geq 0$.

In this case,
\begin{eqnarray*}
\omega'=\sum_{i\in I} x_{i}
c(-1)^{c_{i0}}c(-3)^{c_{i1}}\cdots c(-2l-1)^{c_{il}}\bar{\omega}.
\end{eqnarray*}
We have
\begin{eqnarray*}
&&(d(2l+1)-\psi(d(2l+1)))\omega'\\
&=& \sum_{i\in I} x_{i} c_{il} (2l+1)\xi c(-1)^{c_{i0}}c(-3)^{c_{i1}}\cdots c(-2l-1)^{c_{il}-1} \bar{\omega}
 \neq  0,
\end{eqnarray*}
 a contradiction.
So we deduce that  $\omega'=u\bar{\omega}$ for some $u\in \mathbb{C}^{*}$, with which we prove (1).

(2) follows from (1) and (iii) of Lemma \ref{lemma2.1}.
(3) follows from (2) and (iv) of Proposition \ref{prop2.1}
\end{proof}

\begin{coro}\label{coro3.8}
Let $\psi(c(1))\neq 0$  and $M_{\psi}$ be a  universal Whittaker
module for $\widehat{H}_{4}[\tau]$, generated by the Whittaker
vector $\omega=1\otimes 1$. Then $v\in M_{\psi}$ is a
Whittaker vector if and only if $v=u\omega$ for some $u\in
\mathbb{C}[{\bf k}]$.
\end{coro}

\begin{proof}
Since ${\bf k}$ is the center of $\widehat{H}_{4}[\tau]$, it is easy
to see that $v=u\omega$ is a Whittaker vector if $u\in
\mathbb{C}[{\bf k}]$. Conversely, let $v\in M_{\psi}$ be a Whittaker vector.
By (2.3), we can write
\begin{eqnarray*}
v=\sum\limits_{\tilde{\mu},\nu,\lambda,\eta}p_{\tilde{\mu},\nu,\lambda,\eta}({\bf k})
(a+b)(-\tilde{\mu})(a-b)(-\nu)d(-\lambda)c(-\eta)\omega,
\end{eqnarray*}
where $p_{\tilde{\mu},\nu,\lambda,\eta}({\bf k})\in \mathbb{C}[{\bf k}]$.
There exists a module homomorphism $\varphi
: M_{\psi}\rightarrow L_{\psi,\xi}$ with $\omega \mapsto \bar{\omega}$.
Clearly
\begin{eqnarray*}
\varphi(v)=\sum\limits_{\tilde{\mu},\nu,\lambda,\eta}p_{\tilde{\mu},\nu,\lambda,\eta}(\xi)
(a+b)(-\tilde{\mu})(a-b)(-\nu)d(-\lambda)c(-\eta)\bar{\omega}
\end{eqnarray*}
is a Whittaker vector of $L_{\psi,\xi}$.
Then by Theorem  \ref{thm3.7}, $\varphi(v)\in \mathbb{C}\bar{\omega}$.
Thus $(\tilde{\mu},\nu,\lambda, \eta)=(\bar{0},\bar{0},\bar{0},\bar{0})$ for
any $p_{\tilde{\mu},\nu,\lambda,\eta}({\bf k})\neq 0$, then  $v=u\omega$ for
$u\in \mathbb{C}[{\bf k}]$.
\end{proof}

\section{Whittaker vectors for the case that $\psi(c(1))\neq 0$ and $\xi=0$}
\setcounter{equation}{0}

In this section, we will give all the Whittaker vectors of $L_{\psi,0}$ for which $\psi(c(1))\neq 0$.
 The following is the main result of this section.

\begin{theorem}\label{thm 4.3}
Let $\psi: \widehat{H}_{4}[\tau]^{(+)}\rightarrow \mathbb{C}$ be a Lie
algebra homomorphism such that $\psi(c(1))\neq 0$.   Then
 any Whittaker vector of the Whittaker module $L_{\psi,
0}$ for
$\widehat{H}_{4}[\tau]$  is a non-zero linear combination of elements in $\{\bar{\omega},  c(-\eta)\bar{\omega}\; |\; \eta\in\mathcal {P}_{odd}  \}$.
\end{theorem}

\begin{proof}
Since $\xi=0$, it is easy to see that $\{ c(-\eta)\bar{\omega}\; |\; \eta\in\mathcal {P}_{odd}  \}$
are Whittaker vectors. Denote by $N$ the submodule of $L_{\psi,0}$ generated by $\{ c(-\eta)\bar{\omega}\; |\; \eta\in\mathcal {P}_{odd}  \}$.

It suffices to prove that $\bar{\omega}$ is the only linear independent Whittaker vector of $L_{\psi,0}/N$.
 For convenience, we still denote by $u$ the image of $u\in L_{\psi,0}$
in $L_{\psi,0}/N$. Assume that $u\in L_{\psi,0}/N$ is a Whittaker vector.  Suppose that $u\neq \bar{\omega}$. Then we may assume that
\begin{eqnarray*}
u &=& \sum_{i\in I}x_{i}d(-1)^{c_{i0}}\cdots d(-2m-1)^{c_{im}} (a+b)(0)^{a_{i0}}\cdots (a+b)(-2n)^{a_{in}}\cdot \\
               &&(a-b)(-1)^{b_{i0}}\cdots (a-b)(-2l-1)^{b_{il}}\bar{\omega},
\end{eqnarray*}
where $I$ is a finite set of index, $m,n,l, c_{ij}, a_{ij}, b_{ij}\in \mathbb{N}$, $x_{i}\in \mathbb{C}^*$. For $i\in I$, set

$$
u_i=d(-1)^{c_{i0}}\cdots d(-2m-1)^{c_{im}} (a+b)(0)^{a_{i0}}\cdots (a+b)(-2n)^{a_{in}}
               (a-b)(-1)^{b_{i0}}\cdots (a-b)(-2l-1)^{b_{il}}\bar{\omega}.
$$
Suppose that there exists $i\in I$ such that $c_{im}> 0$. Let $J\in I$ be such that for $i\in J$,
\begin{eqnarray*}
\sum_{k=0}^{m}c_{ik}=\mathop{\operatorname{max}}\limits_{j\in I} \{ \sum_{k=0}^{m}c_{jk}   \}.
\end{eqnarray*}
By considering the action of $(a-b)(2k+1)$ on $u$, if for some $i\in J$ there exists $0\leq k\leq n$ such that $a_{ik}\neq 0$, we can get
$$
2\psi(c(1))x_i a_{ik} d(-1)^{c_{i0}}\cdots d(-2m-1)^{c_{im}} (a+b)(0)^{a_{i0}}\cdots (a+b)(-2k)^{a_{ik}-1} \cdots $$$$
(a+b)(-2n)^{a_{in}}(a-b)(-1)^{b_{i0}}\cdots (a-b)(-2l-1)^{b_{il}}\bar{\omega}=0,
$$
a contradiction. So we have  for $i \in J$,
\begin{eqnarray}\label{equation 4.1}
\sum_{k=0}^{n}a_{ik}=0.
\end{eqnarray}
Similarly for $i\in J$,
\begin{eqnarray}\label{eq4.1}
\sum_{k=0}^{l}b_{ik}=0.
\end{eqnarray}
Let $J_{1}\subset J$ be such that for $i\in J_{1}$, $c_{im}\neq 0$. Let $J_{2}\subset J_{1}$ be such that
\begin{eqnarray*}
c_{im}= \mathop{\operatorname{max}}\limits_{j\in J_{1}} \{ c_{jm}   \}.
\end{eqnarray*}
We may assume that $1\in J_{2}$ and $x_1=2\psi(c(1))$.  By considering the action of $(a+b)(2k+2)$ and $(a-b)(2k+1)$, for $0\leq k\leq m$,
and noticing that
\begin{eqnarray}\label{equation 4.2}
&&(a+b)(2k+2) d(-1)^{c_{10}}\cdots d(-2m-1)^{c_{1m}} \bar{\omega} \nonumber \\
& = & -c_{1m} d(-1)^{c_{10}}\cdots d(-2m-1)^{c_{1m}-1}(a-b)(-2m+2k+1) \bar{\omega}+\cdots.
\end{eqnarray}
\begin{eqnarray}\label{equation 4.3}
&& (a-b)(2k+1)d(-1)^{c_{10}}\cdots d(-2m-1)^{c_{1m}} \bar{\omega} \nonumber \\
& = & -c_{1m} d(-1)^{c_{10}}\cdots d(-2m-1)^{c_{1m}-1} (a+b)(-2m+2k) \bar{\omega}+\cdots.
\end{eqnarray}
We can deduce that there exists $i\in I$ such that
\begin{eqnarray}\label{equation 4.4}
c_{i0}=c_{10}, c_{i1}=c_{11}, \cdots, c_{im-1}=c_{1m-1},  c_{im}=c_{1m}-1, \ \ \sum\limits_{k=0}^na_{ik}+\sum\limits_{k=0}^lb_{ik}\neq 0.
\end{eqnarray}
Let $I_{1}\subset I$ be such that for $i\in I_{1}$,
\begin{eqnarray*}
\{ c_{ij}\; |  \; 0\leq j \leq m  \} \;  \text {satisfies (\ref {equation 4.4})}.
\end{eqnarray*}
Let $I_{2}\subset I_{1}$ be such that for $i\in I_{2}$,
\begin{eqnarray*}
n_{1}=\sum_{k=0}^{n}a_{ik}=\mathop{\operatorname{max}}\limits_{i\in I_{1}} \{ \sum_{k=0}^{n}a_{ik}   \}.
\end{eqnarray*}
Let $I_{2}'\subset I_{1}$ be such that for $i\in I_{2}'$,
\begin{eqnarray*}
l_{1}=\sum_{k=0}^{l}b_{ik}=\mathop{\operatorname{max}}\limits_{i\in I_{1}} \{ \sum_{k=0}^{l}b_{ik}   \}.
\end{eqnarray*}
By (\ref {equation 4.2}), we have $l_{1}> 0$, and by (\ref {equation 4.3}), we have
$n_{1}> 0$. Then
\begin{eqnarray*}
&u&= 2 \psi( c(1) )d(-1)^{c_{10}}\cdots d(-2m-1)^{c_{1m}} \bar{\omega} \\
&&+ \sum_{i\in I_{2}} x_{i}d(-1)^{c_{10}}\cdots d(-2m-1)^{c_{1m}-1}  (a+b)(0)^{a_{i0}}\cdots (a+b)(-2n)^{a_{in}}\cdot\\
               &&   (a-b)(-1)^{b_{i0}}\cdots (a-b)(-2l-1)^{b_{il}}\bar{\omega} \\
&&+ \sum_{i\in I_{2}'} x_{i}d(-1)^{c_{10}}\cdots d(-2m-1)^{c_{1m}-1} (a+b)(0)^{a_{i0}}\cdots (a+b)(-2n)^{a_{in}}\cdot\\
               &&    (a-b)(-1)^{b_{i0}}\cdots (a-b)(-2l-1)^{b_{il}}\bar{\omega} \\
&&+ \sum_{i\in I_{1}  \setminus (I_{2}' \cup I_{2}) } x_{i}d(-1)^{c_{10}}\cdots d(-2m-1)^{c_{1m}-1} (a+b)(0)^{a_{i0}}\cdots (a+b)(-2n)^{a_{in}}\cdot\\
               &&    (a-b)(-1)^{b_{i0}}\cdots (a-b)(-2l-1)^{b_{il}}\bar{\omega} + \sum_{i\in I \setminus I_{1}} x_{i}u_i.
\end{eqnarray*}
If $n_{1}> 2$, let $j\in I_{2}$, $0\leq k \leq n$ be such that $a_{jk} \neq 0$.
Considering the action of $((a-b)(2k+1)-\psi((a-b)(2k+1)))$ on $u$, then the term
\begin{eqnarray*}
&& 2 x_{j} a_{jk} \psi( c(1) ) d(-1)^{c_{10}}\cdots d(-2m-1)^{c_{1m}-1}
 (a+b)(0)^{a_{j0}}\cdots (a+b)(-2k)^{a_{jk}-1}\\
 && \cdots (a+b)(-2n)^{a_{jn}} (a-b)(-1)^{b_{j0}}\cdots (a-b)(-2l-1)^{b_{jl}}\bar{\omega}
\end{eqnarray*}
only appears in
\begin{eqnarray*}
&& (a-b)(2k+1)x_{j} d(-1)^{c_{10}}\cdots d(-2m-1)^{c_{1m}-1}\cdot \\
&& (a+b)(0)^{a_{j0}}\cdots (a+b)(-2k)^{a_{jk}}  \cdots (a+b)(-2n)^{a_{jn}}
 (a-b)(-1)^{b_{j0}}\cdots (a-b)(-2l-1)^{b_{jl}}\bar{\omega}
\end{eqnarray*}
with coefficent $2 x_{j} a_{jk} \psi( c(1) )$. This means that $x_{j}=0$, a contradiction.
So $n_{1}\leq 2$. Similarly, we have $l_{1}\leq 2$.

If for $j\in I_{2}$, there exists  $0\leq k \leq l$ such that
$b_{jk}\neq 0$, considering the action of $(a+b)(2k+2)$ on $u$, we can  deduce that
$x_{j}=0$, a contradiction. So we have
\begin{eqnarray*}
\sum_{k=0}^{l} b_{ik}=0, \; \text {for}\;   i \in I_{2}.
\end{eqnarray*}
Similarly, $\sum\limits_{k=0}^{n} a_{ik}=0$, for  $i \in I_{2}'$. So we have for $i\in I_1$,
\begin{eqnarray*}
0 \leq \sum_{k=0}^{n} a_{ik} + \sum_{k=0}^{l} b_{ik} \leq 2.
\end{eqnarray*}
Suppose that there exists $j\in I_{1}$  such that
\begin{eqnarray*}
u_{j}= d(-1)^{c_{10}}\cdots d(-2m-1)^{c_{1m}-1} (a-b)(-2k-1)\bar{\omega},
\end{eqnarray*}
for some $0 \leq k \leq l$.
If $k\neq m$, consider the action of $(a+b)(2k+2)$ on $u$.
Since
\begin{eqnarray*}
&& (a+b)(2k+2)d(-1)^{c_{10}}\cdots d(-2m-1)^{c_{1m}}\bar{\omega}\\
&=& -c_{1m}d(-1)^{c_{10}}\cdots d(-2m-1)^{c_{1m}-1}(a-b)(2k-2m+1)\bar{\omega}+ \cdots,
\end{eqnarray*}
it follows that the term
\begin{eqnarray*}
d(-1)^{c_{10}}\cdots d(-2m-1)^{c_{1m}-1}\bar{\omega}
\end{eqnarray*}
 appears only  in $x_{j}(a+b)(2k+2) u_{j}$ with coefficient $-2 \psi( c(1) )x_{j}\neq 0$, a contradiction.
So  $k=m$  and
\begin{eqnarray*}
x_{j}=-c_{1m}\psi( (a-b)(1) ).
\end{eqnarray*}
Suppose there exists $j\in I_{1}$  such that
\begin{eqnarray*}
\sum\limits_{k=0}^{n} a_{jk}=1 \; \text{and} \;  \sum\limits_{k=0}^{l} b_{jk}=1.
\end{eqnarray*}
We may
assume that
\begin{eqnarray*}
u_{j}= d(-1)^{c_{10}}\cdots d(-2m-1)^{c_{1m}-1} (a+b)(-2r^{(j)}) (a-b)(-2s^{(j)}-1)\bar{\omega}.
\end{eqnarray*}
We consider the action of $(a+b)(2s^{(j)}+2)$ on $u$, then the term
\begin{eqnarray*}
 d(-1)^{c_{10}}\cdots d(-2m-1)^{c_{1m}-1} (a+b)(-2r^{(j)}) \bar{\omega}
\end{eqnarray*}
appears only in
\begin{eqnarray*}
x_{j}(a+b)(2s^{(j)}+2)u_{j}
\end{eqnarray*}
with coefficient $-2 x_{j} \psi( c(1) )\neq 0$, a contradiction.
So there exists no $j\in I_{1}$  such that
$\sum\limits_{k=0}^{n} a_{jk}=1$,   $\sum\limits_{k=0}^{l} b_{jk}=1$.

Let $j\in I_{1}$ be such that
\begin{eqnarray*}
\sum\limits_{k=0}^{n} a_{jk}=1 \; \text{and} \;  \sum\limits_{k=0}^{l} b_{jk}=0.
\end{eqnarray*}
We may
assume that
\begin{eqnarray*}
u_{j}= d(-1)^{c_{10}}\cdots d(-2m-1)^{c_{1m}-1} (a+b)(-2m^{(j)}) \bar{\omega}.
\end{eqnarray*}
 Consider the action of $(a-b)(2m^{(j)}+1)$ on $u$, then the term
\begin{eqnarray*}
 d(-1)^{c_{10}}\cdots d(-2m-1)^{c_{1m}-1} \bar{\omega}
\end{eqnarray*}
appears only  in
\begin{eqnarray*}
x_{j}(a-b)(2m^{(j)}+1)u_{j}
\end{eqnarray*}
with coefficient $2 \psi( c(1) )x_{j} \neq 0$, a contradiction.
So for $j\in I_{1}$,
if $\sum\limits_{k=0}^{n} a_{jk}\neq 0$,  then $\sum\limits_{k=0}^{n} a_{jk}= 2$  and
$\sum\limits_{k=0}^{l} b_{jk}=0$.

Let $I_4\subseteq I_{1}$ be such that for $j\in I_4$, $\sum\limits_{k=0}^{n} a_{jk}= 2$  and
$\sum\limits_{k=0}^{l} b_{jk}=0$. By (\ref{equation 4.3}), $I_4\neq \emptyset.$
 We may
assume that for $j\in I_4$,
\begin{eqnarray*}
u_{j}= d(-1)^{c_{10}}\cdots d(-2m-1)^{c_{1m}-1} (a+b)(-2m_{1}^{(j)}) (a+b)(-2m_{2}^{(j)})\bar{\omega},
\end{eqnarray*}
and  $0\leq m_{1}^{(j)}\leq m_{2}^{(j)}$.
If $m_{1}^{(j)}+ m_{2}^{(j)} \neq m$, we consider the action of $(a-b)(2m_{2}^{(j)}+1 )-\psi((a-b)(2m_2^{(j)})$ on $u$, then
the term
\begin{eqnarray*}
d(-1)^{c_{10}}\cdots d(-2m-1)^{c_{1m}-1} (a+b)(-2m_{1}^{(j)}) \bar{\omega}
\end{eqnarray*}
appears only in $x_{j}(a-b)(2m_{2}^{(j)}+1 )u_{j}$ with  coefficient
\begin{eqnarray*}
&& 2x_{j} \psi( c(1) )\neq 0, \; \text {if}\;  m_{1}^{(j)} \neq m_{2}^{(j)}, \text{or}\\
&& 4x_{j} \psi( c(1) )\neq 0, \; \text {if}\;  m_{1}^{(j)} = m_{2}^{(j)},
\end{eqnarray*}
a contradiction.
So $m_{1}^{(j)}+m_{2}^{(j)}=m$, that is,
 \begin{eqnarray}\label{equation 4.5}
u_{j}= d(-1)^{c_{10}}\cdots d(-2m-1)^{c_{1m}-1} (a+b)(-2m_{1}^{(j)}) (a+b)(-2m+2m_{1}^{(j)})\bar{\omega}.
\end{eqnarray}
 By (\ref{equation 4.3}),  we have for $j\in I_4$,
 $$x_j=c_{1m}, \ {\rm  if}\  m_{1}^{(j)}\neq m/2, $$
 or
  $$x_j=\frac{1}{2}c_{1m}, \ {\rm if} \  m_{1}^{(j)}=m/2.$$
Let  $I_5 \subseteq I_{1}$ be such that for $j\in I_{5}$,
$\sum\limits_{k=0}^{n} a_{jk}=0$ and $\sum\limits_{k=0}^{l} b_{jk}=2$. By (\ref{equation 4.2}), $I_5\neq \emptyset$, if $m\geq 1$. We may
assume that
\begin{eqnarray*}
u_{j}= d(-1)^{c_{10}}\cdots d(-2m-1)^{c_{1m}-1} (a-b)(-2r_{1}^{(j)}-1) (a-b)(-2r_{2}^{(j)}-1)\bar{\omega}
\end{eqnarray*}
such that $r_{2}^{(j)}\geq r_{1}^{(j)}\geq 0$.
If  $r_{1}^{(j)}+ r_{2}^{(j)} \neq m-1$, we consider the action of $(a+b)(2r_{2}^{(j)}+2 )$ on $u$, then
the term
\begin{eqnarray*}
d(-1)^{c_{10}}\cdots d(-2m-1)^{c_{1m}-1} (a-b)(-2r_{1}^{(j)}-1) \bar{\omega}
\end{eqnarray*}
 appears only in $x_{j}(a+b)(2r_{2}^{(j)}+2 )u_{j}$ with coefficient
\begin{eqnarray*}
&& -2x_{j} \psi( c(1) )\neq 0, \; \text {if}\;  r_{1}^{(j)} \neq r_{2}^{(j)}, \text{or}\\
&& -4x_{j} \psi( c(1) )\neq 0, \; \text {if}\;  r_{1}^{(j)} = r_{2}^{(j)},
\end{eqnarray*}
a contradiction.
Then together with (\ref{equation 4.2}), we have for $j\in I_5$,
\begin{eqnarray}\label{equation 4.6}
u_{j}= d(-1)^{c_{10}}\cdots d(-2m-1)^{c_{1m}-1} (a-b)(-2r_{1}^{(j)}-1) (a-b)(-2m+2r_{1}^{(j)}+1)\bar{\omega},
\end{eqnarray}
for some $0 \leq r_{1}^{(j)} \leq \frac{m-1}{2} $ with coefficient $-c_{1m}$ if $r_{1}^{(j)}\neq \frac{m-1}{2}$,  or $-\frac{1}{2}c_{1m}$ if $r_{1}^{(j)}= \frac{m-1}{2}$.

By the above proof, we have
\begin{eqnarray*}
u &=& 2\psi( c(1)) d(-1)^{c_{10}}\cdots d(-2m-1)^{c_{1m}} \bar{\omega}\\
&& -c_{1m}\psi( (a-b)(1)) d(-1)^{c_{10}}\cdots d(-2m-1)^{c_{1m}-1} (a-b)(-2m-1)\bar{\omega}\\
&& + \sum_{i\in I_{4}}x_{i} d(-1)^{c_{10}}\cdots d(-2m-1)^{c_{1m}-1} (a+b)(-2m_{1}^{(i)}) (a+b)(-2m+2m_{1}^{(i)})\bar{\omega} \\
&& + \sum_{i\in I_{5}}x_{i} d(-1)^{c_{10}}\cdots d(-2m-1)^{c_{1m}-1} (a-b)(-2r_{1}^{(i)}-1) (a-b)(-2m+2r_{1}^{(i)}+1)\bar{\omega} \\
&& + \sum_{i\in I_1\setminus (I_{4}\cup I_5)}x_{i} u_{i}+ \sum_{i\in I\setminus I_1}x_iu_i,
\end{eqnarray*}
where for $j\in I_4$, $0\leq m_{1}^{(j)}\leq m/2$, and $x_j=c_{1m}$   or $\frac{1}{2}c_{1m}$, depending on $m_{1}^{(j)}\neq m/2$ or $m_{1}^{(j)}=m/2$, and for $j\in I_5$, $0 \leq r_{1}^{(j)} \leq \frac{m-1}{2} $,  $x_j=-c_{1m}$ or $-\frac{1}{2}c_{1m}$, depending on $r_{1}^{(j)}\neq \frac{m-1}{2}$ or
$r_{1}^{(j)}= \frac{m-1}{2}$.

We now consider the action of $d(2m+1)$ on $u$, then we have
\begin{eqnarray}\label{equation 4.6}
&&( d(2m+1)-\psi( d(2m+1)) ) u   \nonumber\\
&=& -c_{1m}\psi( (a-b)(1)) d(-1)^{c_{10}}\cdots d(-2m-1)^{c_{1m}-1} (a+b)(0)\bar{\omega} \nonumber \\
&& + \sum_{i\in I_{4}}x_{i} d(-1)^{c_{10}}\cdots d(-2m-1)^{c_{1m}-1}( 2\psi(c(1))+
(a+b)(-2m+2m_{1}^{(i)})\cdot \nonumber\\
&& (a-b)(2m-2m_{1}^{(i)}+1) + (a+b)(-2m_{1}^{(i)}) (a-b)(2m_{1}^{(i)}+1) )\bar{\omega} \nonumber\\
&& + \sum_{i\in I_{5}}x_{i} d(-1)^{c_{10}}\cdots d(-2m-1)^{c_{1m}-1}( -2\psi(c(1))+
(a-b)(-2m+2r_{1}^{(i)}+1)\cdot \nonumber\\
&& (a+b)(2m-2r_{1}^{(i)}) + (a-b)(-2r_{1}^{(i)}-1) (a+b)(2r_{1}^{(i)}+2) )\bar{\omega} \nonumber\\
&& + \sum_{i\in I\setminus I_{1} }( d(2m+1)-\psi( d(2m+1)) ) x_{i} u_{i}=0.\nonumber
\end{eqnarray}
So we have
$$
(\sum\limits_{i\in I_4}x_i-\sum\limits_{j\in I_5}x_j)\psi(c(1))=0.
$$
Since for $i\in I_4, j\in I_5$, $x_i=c_{1m}$ or $\frac{1}{2}c_{1m}$, $x_j=-c_{1m}$ or $-\frac{1}{2}c_{1m}$, and $\psi(c(1))\neq 0$, it follows that
  $c_{1m}=0$, a
 contradiction.
This proves that for all $i\in I$,
\begin{eqnarray*}
c_{i0}=\cdots= c_{im}=0.
\end{eqnarray*}
Then
$$
u=\sum_{i\in I}x_{i} (a+b)(0)^{a_{i0}} \cdots (a+b)(-2n)^{a_{in}}(a-b)(-1)^{b_{i0}} \cdots (a-b)(-2l-1)^{b_{il}} \bar{\omega}
$$
is a Whittaker vector of type $\psi$ of the Weyl algebra, which is linearly spanned by
\begin{eqnarray*}
\{(a+b)(2r), (a-b)(2s+1), c(1)\; |\; r,s \in \mathbb{Z}       \},
\end{eqnarray*}
with relations:
$$
[(a+b)(2r),(a-b)(2s+1)]=-2\delta_{r+s,0}c(1).
$$
Then we immediately have $u=\bar{\omega}$.
\end{proof}

\section{Whittaker vectors for $\psi(c(1))=0$}
\setcounter{equation}{0}

In this section, we assume that $\psi(c(1))= 0$, that is, $\psi$
is singular. We shall continue to investigate the form of Whittaker
vectors and the connections between Whittaker modules and Verma
modules for the twisted affine Nappi-Witten Lie algebra
$\widehat{H}_{4}[\tau]$.

We still follow the notations in Theorem \ref{thm3.7} and Corollary \ref{coro3.8}. In particular,
$\bar{\omega}$ (resp. $\omega$) denotes the cyclic Whittaker
vector $\overline{1\otimes 1}$ (resp. $1\otimes 1$) for $L_{\psi, \xi}$
(resp. $M_{\psi}$).

Note that $\psi((a+b)(n))=\psi((a-b)(m))=\psi(c(l))=0$ for $n\in
2\mathbb{Z}_{+}$, $m\in 2\mathbb{Z}_{+}+1$, $l\in 2\mathbb{N}+1$.
Denote $\psi((a-b)(1))$  by $\sigma_{1}$, we have the following
lemma by straightforward computations.

\begin{lem}
\begin{itemize}
\item[{\rm(i)}] If $\sigma_{1}=0$, then $(a+b)(0)\bar{\omega}$ is a
Whittaker vector of $L_{\psi, \xi}$;
\item[{\rm(ii)}]  If $\sigma_{1}\neq
0$, then $(\xi(a+b)(0)-\sigma_{1}c(-1))\bar{\omega}$ is a Whittaker
vector of $L_{\psi, \xi}$.
\end{itemize}
\end{lem}

The following theorem gives a full characterization of Whittaker vectors of $L_{\psi,\xi}$ for $\psi(c(1))=0$, $\xi \neq 0$.
\begin{theorem}\label{prop5.3}
Let $\psi(c(1))=0$, $\xi \neq 0$.    Set
\begin{eqnarray*}
z=\left\{
\begin{array}{ll}
(a+b)(0), \quad \quad\quad\quad\quad\quad\;  if\; \sigma_{1}=0, \\
\xi(a+b)(0)-\sigma_{1}c(-1),\quad   if \;  \sigma_{1}\neq0,
\end{array}
\right.
\end{eqnarray*}
then $v$ is a Whittaker
vector of $L_{\psi, \xi}$ if and only if $v=u\bar{\omega}$ for some $u\in \mathbb{C}[z]$, where $\mathbb{C}[z]$ is the
polynomial algebra generated by $z$.

\end{theorem}

\begin{proof}
By Lemma 5.1 and  induction on the degree of $u$ as a polynomial, it is easy to check
 that $\mathbb{C}[z]\bar{\omega}$ are Whittaker
vectors of $L_{\psi, \xi}$.
Conversely, let $v$ be a  Whittaker
vector of $L_{\psi, \xi}$. Note that for $m,n\in 2\mathbb{N}+1$,
\begin{eqnarray*}
[c(m),d(-n)]=m\delta_{m,n}{\bf k},
\end{eqnarray*}
then
 \begin{eqnarray*}
 {\frak s}=\mathop{\bigoplus}\limits_{m\in 2{\mathbb N}+1}(\mathbb{C} c(m)\oplus\mathbb{C}
d(-m))\oplus \mathbb{C}{\bf k}
\end{eqnarray*}
is a Heisenberg algebra and $L_{\psi, \xi}$ can be viewed as an ${\frak s}$-module
such that ${\bf k}$ acts as $\xi \neq 0$.
Since every highest weight ${\frak s}$-module
generated by one element with $\k$ acting as a non-zero scalar is
irreducible, it follows  that $L_{\psi, \xi}$ can be decomposed into a direct
sum of irreducible highest weight modules of ${\frak s}$ with
the highest weight vectors
 \begin{eqnarray*}
\{(a+b)(-\tilde{\mu})(a-b)(-\nu)c(-\eta)\bar{\omega}\; | \;  (\tilde{\mu},\nu,\eta)\in \widetilde{\mathcal
{P}}_{even}\times \mathcal {P}_{odd}\times \mathcal {P}_{odd}\}.
\end{eqnarray*}
So if $\bar{\omega}\neq v \in L_{\psi, \xi}$ is a Whittaker
vector, then $v$ is a linear combination of elements of the form
\begin{eqnarray*}
(a+b)(-\tilde{\mu})(a-b)(-\nu)c(-\eta)\bar{\omega},\;   (\tilde{\mu},\nu,\eta)\in \widetilde{\mathcal
{P}}_{even}\times \mathcal {P}_{odd}\times \mathcal {P}_{odd}.
\end{eqnarray*}
We may assume that
\begin{eqnarray*}
&v=&\sum_{i\in I} x_{i} (a+b)(0)^{a_{i0}}(a+b)(-2)^{a_{i1}}\cdots (a+b)(-2n)^{a_{in}} \cdot\\
&&(a-b)(-1)^{b_{i0}}(a-b)(-3)^{b_{i1}} \cdots (a-b)(-2m-1)^{b_{im}} c(-1)^{c_{i0}}c(-3)^{c_{i1}}\cdots c(-2l-1)^{c_{il}}\bar{\omega},
\end{eqnarray*}
where $I$ is a finite index set, $n,m,l,a_{ij},b_{ij},c_{ij}\in \mathbb{N}$.

{\bf Case I}  Suppose $n> 0$ and $\{ a_{in} \; | \;  i\in I\}\neq \{0\}$.

If $m\geq n$ and there exists $i\in I$ such that $b_{im}\neq 0$, we obtain
\begin{eqnarray*}
&&((a-b)(2m+1)-\psi((a-b)(2m+1)))v\\
&=& \sum_{i\in I} x_{i} (-2)(2m+1)\xi b_{im} (a+b)(0)^{a_{i0}}(a+b)(-2)^{a_{i1}}\cdots (a+b)(-2n)^{a_{in}}\cdot\\
&&(a-b)(-1)^{b_{i0}}(a-b)(-3)^{b_{i1}}\cdots (a-b)(-2m-1)^{b_{im}-1} c(-1)^{c_{i0}}c(-3)^{c_{i1}}\cdots c(-2l-1)^{c_{il}} \bar{\omega}\\
 &\neq & 0,
\end{eqnarray*}
 a contradiction. Now assume $m< n$ and there exists $i\in I$ such that $b_{im}\neq 0$, we deduce
\begin{eqnarray*}
&&((a+b)(2n)-\psi((a+b)(2n)))v\\
&=& \sum_{i\in I} x_{i} 2(2n)\xi a_{in} (a+b)(0)^{a_{i0}}(a+b)(-2)^{a_{i1}}\cdots (a+b)(-2n)^{a_{in}-1}\cdot\\
&&(a-b)(-1)^{b_{i0}}(a-b)(-3)^{b_{i1}}\cdots (a-b)(-2m-1)^{b_{im}} c(-1)^{c_{i0}}c(-3)^{c_{i1}}\cdots c(-2l-1)^{c_{il}} \bar{\omega}\\
 &\neq & 0,
\end{eqnarray*}
 a contradiction again.

{\bf Case II}  Suppose $m\geq 0$ and $\{ b_{im} \; | \;  i\in I\}\neq \{0\}$ and $\{ a_{in} \; | \;  i\in I\}= \{0\}$ for any $n> 0$.

Then $v$ has the following form
\begin{eqnarray*}
v&=&\sum_{i\in I} x_{i} (a+b)(0)^{a_{i0}}
(a-b)(-1)^{b_{i0}}(a-b)(-3)^{b_{i1}} \cdots (a-b)(-2m-1)^{b_{im}} \cdot \\
&&c(-1)^{c_{i0}}c(-3)^{c_{i1}}\cdots c(-2l-1)^{c_{il}}\bar{\omega}.
\end{eqnarray*}
It is easy to check that
$((a-b)(2m+1)-\psi((a-b)(2m+1)))v\neq0$, a contradiction.

{\bf Case III}  Suppose $l\geq 1$ and $\{ c_{il} \; | \;  i\in I\}\neq \{0\}$ and $\{ a_{in} \; | \;  i\in I\}=\{ b_{im} \; | \;  i\in I\}= \{0\}$ for any $n> 0, m\geq 0$.

In this case,
\begin{eqnarray*}
v=\sum_{i\in I} x_{i} (a+b)(0)^{a_{i0}}
c(-1)^{c_{i0}}c(-3)^{c_{i1}}\cdots c(-2l-1)^{c_{il}}\bar{\omega}.
\end{eqnarray*}
We have
\begin{eqnarray*}
&&(d(2l+1)-\psi(d(2l+1)))v\\
&=& \sum_{i\in I} x_{i} (2l+1)\xi c_{il} (a+b)(0)^{a_{i0}}c(-1)^{c_{i0}}c(-3)^{c_{i1}}\cdots c(-2l-1)^{c_{il}-1} \bar{\omega}\\
 &\neq & 0,
\end{eqnarray*}
 a contradiction.

 If $\sigma_{1}=0$, by Case III, it also leads to a contradiction for $l\geq 0$. Then
\begin{eqnarray*}
v=\sum_{i\in I} x_{i} (a+b)(0)^{a_{i0}}
\bar{\omega},
\end{eqnarray*}
 and $v\in \mathbb{C}[(a+b)(0)]\bar{\omega}$.
If $\sigma_{1}\neq 0$, then
\begin{eqnarray*}
v=\sum_{i\in I} x_{i} (a+b)(0)^{a_{i0}}
c(-1)^{c_{i0}}\bar{\omega}.\label{5.1}
\end{eqnarray*}
 For any $n\in 2\mathbb{Z}_{+}, m\in 2\mathbb{N}+1$, it is easy to see that
\begin{eqnarray*}
((a+b)(n)-\psi((a+b)(n)))v=((a-b)(m)-\psi((a-b)(m)))v=(c(m)-\psi(c(m)))v=0.
\end{eqnarray*}
We consider $(d(m)-\psi(d(m)))v$ for $m\in 2\mathbb{N}+1$. Then we have
\begin{eqnarray*}
&&(d(m)-\psi(d(m)))v   \\
&=&\sum_{i\in I} x_{i}( [d(m), (a+b)(0)^{a_{i0}}] c(-1)^{c_{i0}}\bar{\omega}+ (a+b)(0)^{a_{i0}}[d(m),c(-1)^{c_{i0}} ]\bar{\omega}) \nonumber \\
&=& \sum_{i\in I} x_{i} (a_{i0}(a-b)(m) (a+b)(0)^{a_{i0}-1}c(-1)^{c_{i0}} \bar{\omega}+m \delta_{m,1}\xi c_{i0}(a+b)(0)^{a_{i0}} c(-1)^{c_{i0}-1}\bar{\omega}).
\end{eqnarray*}
By the above equtation, if $m>1$, $(d(m)-\psi(d(m)))v =0$. If $m=1$,  we have
\begin{eqnarray}
&&(d(1)-\psi(d(1)))v  \nonumber \\
&= & \sum_{i\in I} x_{i} ( a_{i0}(a-b)(1) (a+b)(0)^{a_{i0}-1}c(-1)^{c_{i0}} + \xi c_{i0} (a+b)(0)^{a_{i0}} c(-1)^{c_{i0}-1})\bar{\omega}\nonumber\\
& = & 0, \label{5.2}
\end{eqnarray}
as $v$ is a Whittaker vector.

In the following we prove that $ v=\sum\limits_{i\in I} x_{i} (a+b)(0)^{a_{i0}}
c(-1)^{c_{i0}}\bar{\omega}\in \mathbb{C}[z]\bar{\omega}$, where $z=\xi(a+b)(0)-\sigma_{1}c(-1)$.
We may assume that
\begin{eqnarray*}
c_{10}=\mathop{\operatorname{max}}\limits_{i\in I}\{ c_{i0} \}.
\end{eqnarray*}
By (\ref{5.2}), $a_{10}=0$.  Then we have
\begin{eqnarray*}
v&=&\sum_{i\in I} x_{i} (a+b)(0)^{a_{i0}}
c(-1)^{c_{i0}}\bar{\omega}\\
&=&  x_{1}'c(-1)^{c_{10}}\bar{\omega}+\sum_{\substack{i\in I \\ c_{i0}<c_{10}}} x_{i} (a+b)(0)^{a_{i0}}c(-1)^{c_{i0}}\bar{\omega}\\
&=& x_{1}'\sigma_{1}^{-c_{10}}(\sigma_{1}c(-1)-\xi (a+b)(0))^{c_{10}}\bar{\omega}+\sum_{\substack{i\in I \\ g_{i0}<c_{10}}} y_{i} (a+b)(0)^{h_{i0}}c(-1)^{g_{i0}}\bar{\omega}.
\end{eqnarray*}
Denote $\sum\limits_{\substack{i\in I \\ g_{i0}<c_{10}}} y_{i} (a+b)(0)^{h_{i0}}c(-1)^{g_{i0}}\bar{\omega}$  by $v'$. Then $v'$ is also a Whittaker vector. By inductive assumption, $v' \in \mathbb{C}[z]\bar{\omega}$. Thus $v\in \mathbb{C}[z]\bar{\omega}$.
\end{proof}

\begin{rem}
Suppose  $\psi(c(1))=0$. Set
\begin{eqnarray*}
z=\left\{
\begin{array}{ll}
(a+b)(0), \quad \quad\quad\quad\quad\quad\;  if\; \sigma_{1}=0, \\
(a+b)(0){\bf k}-\sigma_{1}c(-1),\quad   if \;  \sigma_{1}\neq0.
\end{array}
\right.
\end{eqnarray*}
The same argument as in Theorem \ref{prop5.3}, by  replacing $\xi$ with ${\bf k}$ and $\bar{\omega}$ with
$\omega$ whenever necessary, proves that $v$ is a Whittaker
vector of $M_{\psi}$ if and only if $v=u\omega$ for some $u\in \mathbb{C}[z, {\bf k}]$,
where $\mathbb{C}[z, {\bf k}]$ is the
polynomial algebra generated by $z$ and ${\bf k}$.
\end{rem}

In the following we consider the case that  $\psi$ is identically zero. We first recall the definition of Verma module for the  twisted affine
Nappi-Witten Lie algebra $\widehat{H}_{4}[\tau]$ given in \cite{CJJ}. If
$\psi$ is identically zero, for  $l\in \mathbb{C}$, let
\begin{eqnarray*}
M_{l}=U(\widehat{H}_{4}[\tau])((a+b)(0)-l)\bar{\omega},
\end{eqnarray*}
then the
quotient module
$$M(\xi,l):= L_{\psi,\xi}/ M_{l}$$
is a Verma module for $\widehat{H}_{4}[\tau]$. By Theorem 2.1 of \cite{CJJ}, we
immediately obtain the following lemma.

\begin{lem}  {\rm (\cite{CJJ})} \label{lem 5.4}
For $\xi, l \in \mathbb{C}$, the Verma module $M(\xi,l)$ of
$\widehat{H}_{4}[\tau]$ is irreducible if and only if $\xi \neq 0$.
\end{lem}

\begin{theorem}\label{thm5.5}
If $\psi$ is identically zero and $\xi\neq 0$, for each $l\in \mathbb{C}$ and $i\in\Z_{\geq 1}$, define
 $$M^{i}=U(\widehat{H}_{4}[\tau])((a+b)(0)-l)^{i}\bar{\omega}.$$
Then
\begin{itemize}
\item[{\rm(i)}] $M^{i}$ is a Whittaker submodule of $L_{0,\xi}$, with a cyclic Whittaker vector
$\omega_{i}=((a+b)(0)-l)^{i}\bar{\omega}$, and $M^{i+1}$ is a maximal submodule of $M^{i}$;
\item[{\rm(ii)}] $L_{0,\xi}   \cong M^{i}$ for any $i\in \mathbb{N}$;
\item[{\rm(iii)}]  $M^{i}/ M^{i+1}\cong M(\xi,l)$ with the multiplicity space consisting of polynomials in one variable;
\item[{\rm(iv)}] $L_{0,\xi}$ has a filtration
$$L_{0,\xi}=M^{0}\supseteq M^{1}\supseteq \cdots \supseteq M^{i}\supseteq \cdots$$
with the simple sections given by the Verma module $M(\xi,l)$ with multiplicity infinity.
\end{itemize}
\end{theorem}

\begin{proof} For (i), by Theorem \ref{prop5.3},  $((a+b)(0)-l)^{i}\bar{\omega}$ is a
Whittaker vector of $L_{0,\xi}$, thus $M^{i}$ is a Whittaker
module.

For (ii), we define the linear map for $i\in \mathbb{N}$,
\begin{eqnarray*}
\phi:  & & L_{0,\xi} \rightarrow  M^{i}\\
& &  u\bar{\omega} \mapsto  u((a+b)(0)-l)^{i}\bar{\omega},
\end{eqnarray*}
where $u\in U(\widehat{H}_{4}[\tau]^{(-)}\oplus \mathbb{C}(a+b)(0))$.
 It is easy to check that $\phi$ is an isomorphism of modules. Thus
 $L_{0,\xi}   \cong M^{i}$.

 For (iii), it is immediate that the linear map
 \begin{eqnarray*}
\phi:  & & L_{0,\xi} \rightarrow  M^{i}/ M^{i+1}\\
& &  u\bar{\omega} \mapsto  u((a+b)(0)-l)^{i}\bar{\omega},
\end{eqnarray*}
is an epimorphism of modules and
$Ker\; \phi =  M^{1} $. Thus $ M^{i}/ M^{i+1}\cong L_{0,\xi}/
M^{1}= M(\xi,l)$. Then $M^{i+1}$ is a maximal submodule of $M^{i}$
by Lemma \ref{lem 5.4}.

(iv) follows from (iii) and Lemma \ref{lem 5.4}.
\end{proof}

\begin{theorem}
If $\psi$ is identically zero and $\xi=0$, then
$\{d(-\lambda)\bar{\omega} \mid \lambda \in \mathcal {P}_{odd}\}$
generates a maximal non-trivial submodule $U$ of $L_{0,0}$.
Furthermore, $L_{0,0}/U$ is isomorphic to the one-dimensional
$\widehat{H}_{4}[\tau]$-module $\mathbb{C}\bar{\omega}$.
\end{theorem}

\begin{proof} By  the fact that $\psi\equiv 0$ and the definition of the submodule
$U$, we see that $$(a+b)(-n)\bar{\omega}, \ (a-b)(-m)\bar{\omega}, \
c(-m)\bar{\omega}, \ d(-m)\bar{\omega}\in U$$ for all $n\in
2\mathbb{N}$ and $m\in 2\mathbb{N}+1$. Thus
$$(a+b)(-\tilde{\mu})(a-b)(-\nu)d(-\lambda)c(-\eta)\bar{\omega}\in U$$
for all $(\tilde{\mu},\nu,\lambda,\eta)\in \widetilde{\mathcal
{P}}_{even}\times \mathcal {P}_{odd}\times \mathcal {P}_{odd}\times
\mathcal {P}_{odd}$ with $\#(\tilde{\mu},\nu,\lambda,\eta)>0$. Since
${\bf k}\bar{\omega}=\widehat{H}_{4}[\tau]^{(+)}\bar{\omega}=0$, it follows that
each element of $U$ can be written as a linear combination of elements
of the form
$(a+b)(-\tilde{\mu})(a-b)(-\nu)d(-\lambda)c(-\eta)\bar{\omega}$ with
 $\#(\tilde{\mu},\nu,\lambda,\eta)>0$. Moreover, $\bar{\omega}\notin
U$ as $\xi=0$. Therefore, $L_{0,0}/U$ is isomorphic to the one-dimensional
$\widehat{H}_{4}[\tau]$-module $\mathbb{C}\bar{\omega}$ and $U$ is a
maximal non-trivial submodule of $L_{0,0}$.
\end{proof}
\begin{rem}
If $\psi$ is identically zero and $\xi=0$, then by  Lemma \ref{lem 5.4}
$$M(0,l)= L_{0,0}/ M_{l}$$
is reducible, where $M_{l}=U(\widehat{H}_{4}[\tau])((a+b)(0)-l)\bar{\omega}$. Furthermore, denote by $\bar{\bar{\omega}}$
the image of $\bar{\omega}\in L_{0,0}$ in $L_{0,0}/ M_{l}$. By  Theorem 2.1 of {\rm \cite{CJJ}}, we have
\begin{itemize}
\item[{\rm(i)}] if $l=0$, $\bar{U}=\langle \{d(-\lambda)\bar{\bar{\omega}} \mid \lambda \in \mathcal {P}_{odd}\} \rangle$ is a maximal non-trivial submodule of $M(0,l)$;
\item[{\rm(ii)}] if $l\neq0$,   $\bar{U}'=\langle \{c(-\eta)\bar{\bar{\omega}} \mid \eta \in \mathcal {P}_{odd}\} \rangle$ is a  maximal non-trivial submodule of $M(0,l)$.
\end{itemize}
\end{rem}


\begin{thebibliography}{FKR}


\bibitem{ALZ}
D. Adamovi$\acute{c}$, R. Lu, K. Zhao, Whittaker modules for the affine Lie algebra $A_{1}^{(1)}$,
Adv. Math.  289 (2016) 438-479.


\bibitem{AP}
 D. Arnal, G. Pinczon, On algebraically irreducible
representations of the Lie algebra $sl(2)$,  J. Math. Phys.
15 (1974) 350-359.


\bibitem{B}
R. Block, The irreducible representations of the Lie algebra
$sl(2)$ and of the Weyl algebra,   Adv. Math.   39 (1) (1981) 69-110.




\bibitem{BJP}
Y. Bao, C. Jiang, Y. Pei,  Representations of affine
Nappi-Witten algebras,  J. Algebra  342  (2011)  111-133.





\bibitem{BM}
P. Batra, V. Mazorchuk,  Blocks and modules for Whittaker pairs,   J. Pure Appl. Algebra    215 (2011) 1552-1568.






\bibitem{BO}
G. Benkart, M. Ondrus,  Whittaker modules for generalized Weyl
algebras,  Represent. Theory   13  (2009) 141-164.


\bibitem{C}
K. Christodoulopoulou,  Whittaker modules for Heisenberg algebras
and imaginary Whittaker modules for affine Lie algebras,  J.
Algebra    339  (2008) 2871-2890.











\bibitem{CJJ}
X. Chen, C. Jiang,  Q. Jiang, Representations of the twisted
affine Nappi-Witten algebras,  J. Math. Phys.
54 (5) (2013) 051703, 20 pp.





\bibitem{D}
J. Distler,  talk at the ``Strings 93" Conference, Berkeley,
May 1993.




\bibitem{JW}
C. Jiang,  S. Wang,   Extension of vertex operator algebra
$V_{\widehat{H}_{4}}(l,0)$,  Algebra Colloq.   21 (3) (2014)
361-380.


\bibitem{K1} B. Kostant, On Whittaker vectors and representation
theory, Invent. Math.  48 (1978) 101-184.


\bibitem{K}
V.G. Kac,  Infinite-Dimensional Lie Algebras, third ed.,
Cambridge University Press, Cambridge, UK, 1990.


\bibitem{KK}
E. Kiritsis, C. Kounnas, String propagation in gravitational wave
backgrounds,   Phys. Lett. B  594 (1994)  368-374.



\bibitem{LWZ}
D. Liu, Y. Wu, L. Zhu,  Whittaker modules for the twisted
Heisenberg-Virasoro algebra,   J.  Math.  Phys.
51 (2) (2010) 023524, 12 pp.





\bibitem{NW}
C. Nappi, E. Witten,  Wess-Zumino-Witten model based on a
nonsemisimple group,   Phys. Rev. Lett.  23  (1993)  3751-3753.


\bibitem{O}
 M. Ondrus, Whittaker modules for $U_{q}(sl_{2})$,  J. Algebra   289  (2005)
192-213.




\bibitem{OW}
M. Ondrus, E. Wiesner, Whittaker modules for the Virasoro
algebra, J. Algebra Appl.  8 (3) (2009) 363-377.



\bibitem{S}
A. Sevostyanov, Quantum deformation of Whittaker modules and Toda
lattice,   Duke Math. J.   105 (2000) 211-238.




\bibitem{W}
E. Witten, Non-abelian bosonization in two-dimensions, Commun.
Math. Phys.   92 (1984) 455-472.


\bibitem{ZTL}
X. Zhang, S. Tan, H. Lian, Whittaker modules for the
Schr\"{o}dinger-Witt algebra,   J.  Math. Phys.
 51 (8) (2010) 083524, 17 pp.



\end{thebibliography}
\end{document}